\documentclass[12pt]{article}

\usepackage{amsfonts,amsmath,amssymb,amsthm}
\usepackage[T1]{fontenc}
\usepackage{fullpage}
\usepackage{array}
\usepackage{enumerate}

\usepackage{graphicx}
\usepackage[all]{xy}

\usepackage[colorlinks=true]{hyperref}

\def\strutdepth{\dp\strutbox}
\def \ss{\strut\vadjust{\kern-\strutdepth \sss}}
\def \sss{\vtop to \strutdepth{
\baselineskip\strutdepth\vss\llap{$\diamondsuit\;\;$}\null}}

\def\strutdepth{\dp\strutbox}
\def \sst{\strut\vadjust{\kern-\strutdepth \ssss}}
\def \ssss{\vtop to \strutdepth{
\baselineskip\strutdepth\vss\llap{$\spadesuit\;\;$}\null}}

\def\strutdepth{\dp\strutbox}
\def \ssh{\strut\vadjust{\kern-\strutdepth \sssh}}
\def \sssh{\vtop to \strutdepth{
\baselineskip\strutdepth\vss\llap{$\heartsuit\;\;$}\null}}

\newtheorem{theorem}{Theorem}[section]
\newtheorem{proposition}[theorem]{Proposition}
\newtheorem{lemma}[theorem]{Lemma}
\newtheorem{question}[theorem]{Question}

\theoremstyle{definition}

\theoremstyle{remark}
\newtheorem{remark}[theorem]{Remark}

\newcommand{\N}{\mathbb{N}}

\newcommand{\Q}{\mathbb{Q}}
\newcommand{\Z}{\mathbb{Z}}
\newcommand{\Hyp}{\mathbb{H}}
\newcommand{\FN}{F_N}
\newcommand{\BFN}{\overline{F}_N}
\newcommand{\Aut}{\text{Aut}}
\newcommand{\Inn}{\text{Inn}}
\newcommand{\Out}{\text{Out}}
\newcommand{\Fix}{\text{Fix}}
\newcommand{\Per}{\text{Per}}

\newcommand{\rank}{\text{rank}}

\renewcommand {\epsilon}{\varepsilon}
\renewcommand {\phi}{\varphi}

\renewcommand {\leq}{\leqslant}
\renewcommand {\geq}{\geqslant}

\title{Free group automorphisms with parabolic boundary orbits}
\author{Arnaud Hilion
\footnote{LATP - Universit\'e Aix-Marseille 3 - 
Avenue de l'escadrille Normandie-Ni\'emen - 
13397 Marseille Cedex 20 - France --
email: {\tt arnaud.hilion@univ.u-3mrs.fr}}
}

\date{\today}

\begin{document}

\maketitle

\begin{abstract}
For $N\geq 4$, we show that there exist automorphisms of the free group
$\FN$ which have a parabolic orbit in $\partial\FN$.
In fact, we exhibit a technology for producing infinitely many
such examples.
\end{abstract}


\section{Introduction}

An automorphism $\phi$ of the free group $\FN$ of rank $N$ 
induces a homeomorphism $\partial\phi$ of
the (Gromov) boundary $\partial\FN$ of $\FN$.
The dynamics of the map $\partial\phi$ on $\partial\FN$
has been studied a lot, see \cite{LL1,LL2,LL4,LL3,H0}. 
We give a survey of the known results relevant in our context
in \S \ref{sec:basics}.
In this paper, we focus on the following question:

\bigskip

{\em Does there exist an automorphism $\phi$ of $\FN$
such that there is a
parabolic orbit for the homeomorphism
$\partial\phi$?}

\bigskip
 
We say that an 
automorphism $\phi$ {\em has a parabolic orbit}
if there exists two points $X,Y\in\partial\FN$, $X\neq Y$, such that:
$$\lim_{k\rightarrow\pm\infty}\partial\phi^k(Y)=X.$$
We note that this implies that $X$ is a fixed point of $\partial\phi$.
In such a situation, the point $X\in\partial\FN$ is called a
{\em parabolic fixed point} for $\phi$, and the set
$\{\partial\phi^k(Y)\; | \;k\in\Z\}$ is called a 
{\em parabolic orbit}
for $\phi$.
We prove:

\begin{theorem}\label{thm:main}
For $N\geq 4$
there exists an infinite family $\{\phi_k\;|\;k\in\N\}$ of
automorphisms of $\FN$ which have a parabolic orbit,
such that for any $k,k',p,p'\in\N$, $\phi_k^p$ and
$\phi_{k'}^{p'}$ are conjugated if and only if $k=k'$
and $p=p'$.
\end{theorem}

Discussions with some of the experts of the subject have led the
author to feel that
the existence of such parabolic orbits come somehow as a surprise. 
To put Theorem \ref{thm:main} in prospective, we would like to 
mention the following three facts.

First, 
given a compact set $K$ and a homeomorphism $f$
of $K$, one says that $f$ has {\em North-South dynamics},
if {\em (i)} $f$ has precisely two distinct fixed points $x^+$ and $x^-$, 
{\em (ii)} $\lim_{k\rightarrow+\infty}f^{k}(y) = x^{+}$
and $\lim_{k\rightarrow+\infty}f^{-k}(y) = x^{-}$
for all $y\in K\smallsetminus\{x^-,x^+\}$,
and {\em (iii)} the limit of
$f^{k}$, when $k$ tends to infinity,
is uniform on compact subsets of 
$K\smallsetminus\{x^{-}\}$
and the limit of
$f^{-k}$ is uniform on compact subsets of 
$K\smallsetminus\{x^{+}\}$.
It is proved in \cite{LL1} that ``most'' automorphisms of $\FN$,
in a precise sense we do not explain here, 
have North-South dynamics on $\partial\FN$.
In particular, they can not have a parabolic orbit.

Second, let $\delta$ be the automorphism of 
$F_2=<a,b>$ defined by $\delta(a)=a$ and $\delta(b)=ba$.
The outer automorphism class $D$ of $\delta$ is sometimes called a
{\em Dehn twist automorphism}.
The reader, who has in mind the action 
by isometries of $SL_2(\Z)$ on the hyperbolic plane,
should be warned that Dehn twist automorphisms 
do not give rise to parabolic orbits
in $\partial F_2$.
We give in \S \ref{sec:dehn} a description of all possible 
dynamics of automorphisms of $F_2$ in the outer class
$D^n$, for $n\in\Z$.

Third, more generally, it is known
that geometric automorphisms of $\FN$ do
not have parabolic orbits in $\partial\FN$.
We recall that an automorphisms $\phi$ of $\FN$ is {\em geometric}
if  there exist a surface $S$ (with non empty boundary)
with fundamental group $\pi_1(S)$ isomorphic
to $\FN$ and a homeomorphism $f$ of $S$
which induces $\phi$ on $\FN\cong \pi_1(S)$.
More details are given in \S \ref{sec:geometric}. 
As a consequence, since all automorphisms of $F_2$ are known to be geometric, one obtains:

\begin{proposition}\label{prop:F2}
There does not exist an automorphism of $F_2$
which has a parabolic orbit.
\end{proposition}

To our knowledge, the question of the existence
of automorphisms with a parabolic orbit is still open for $F_3$.

\bigskip

{\em Acknowledgments.}
I would like to express my gratitude to Gilbert Levitt, who has posed the question
of the existence of parabolic orbits as part of my thesis project, see\cite{H0}, 
and has consistently encouraged me to publish my results since then.
I would like to thank Pascal Hubert and Erwan Lanneau for helpful 
discussions about dilatation coefficients of matrices in $SL_2(\Z)$.
I am grateful to Martin Lustig for his active interest in the present paper.


\section{A first example}\label{sec:first}

For the impatient reader, we give a first example of an automorphism
of $F_4=\;<a,b,c,d>$ with a parabolic orbit ``inside $F_4$''
(using Proposition \ref{prop:omega}, this gives immediately a parabolic
orbit in $\partial F_4$).

Let $\phi$ be the automorphism defined by:
$$
\begin{array}{crcl}
\phi: & a & \mapsto & a \\  
& b & \mapsto & ba \\ 
& c & \mapsto & ca^2 \\ 
& d & \mapsto & dc.
\end{array}
$$ 

The inverse of $\phi$ is given by: 
$$
\begin{array}{crcl}
\phi^{-1}: & a & \mapsto & a \\  
& b & \mapsto & ba^{-1} \\ 
& c & \mapsto & ca^{-2} \\ 
& d & \mapsto & da^{2}c^{-1}.
\end{array}
$$

The common limit point of the forward and backward iteration of
$\phi$ (called a ``parabolic fixed point'') will be the element:
$ba^{-\infty}=ba^{-1}a^{-1}a^{-1}a^{-1}\dots\in\partial F_4.$
The element of $F_4$ which gives rise to a parabolic orbit with 
this limit point is $bd^{-1}$.
We calculate:
$$bd^{-1}\stackrel{\phi}{\mapsto}bac^{-1}\cdot d^{-1}
\stackrel{\phi}{\mapsto}bc^{-1}\cdot c^{-1}d^{-1}
\stackrel{\phi}{\mapsto}ba^{-1}c^{-1}\cdot a^{-2}c^{-1}c^{-1}d^{-1}
\stackrel{\phi}{\mapsto}ba^{-2}c^{-1}\cdot a^{-4}c^{-1}a^{-2}c^{-1}c^{-1}d^{-1}
\stackrel{\phi}{\mapsto}\dots
$$
$$b\cdot d^{-1}\stackrel{\phi^{-1}}{\mapsto}ba^{-1}\cdot ca^{-2}d^{-1}
\stackrel{\phi^{-1}}{\mapsto}ba^{-2}\cdot ca^{-4}ca^{-2}d^{-1}
\stackrel{\phi^{-1}}{\mapsto}ba^{-3}\cdot ca^{-6}ca^{-4}ca^{-2}d^{-1}
\stackrel{\phi^{-1}}{\mapsto}\dots
$$ 
In these calculations, we 
help the reader to follow through the iteration by introducing
an extra $\cdot$ which is ``mapped'' to the $\cdot$ in the next iteration
step. The crucial feature is that at any of these $\cdot$ no cancellation
does occur.
We see that $\lim_{k\rightarrow+\infty}\phi^k(bd^{-1})=
\lim_{k\rightarrow+\infty}\phi^{-k}(bd^{-1})=ba^{-\infty}$.
A more 
formal justification is given in \S \ref{sec:parabolic}.


\section{Basics}\label{sec:basics}

This section serves sort of as glossary: We summarize in a sequence of brief subsections the basic definitions and facts which are needed to follow the arguments in the subsequent sections. The expert reader is encouraged to skip the first few subsection (and to go back later to them, if need be). However, the terminology introduced in the last subsections is non-standard and should be read carefully.

\subsection{The induced boundary homeomorphism}

Let $\FN$ denote the free group of finite rank $N\geq2$. 
The boundary $\partial \FN$ of $\FN$ is a Cantor set.
If $\mathcal{A}=\{a_1,\dots,a_N\}$ is a basis of $\FN$,
we denote by $\mathcal{A}^{\pm 1}$ the set 
$\{a_1,\dots,a_N,a_1^{-1},\dots,a_N^{-1}\}$.
A word $w=w_1\dots w_p$ ($w_i\in\mathcal{A}^{\pm 1}$)
is {\em reduced} if $w_{i+1}\neq w_i^{-1}$.
The free group $\FN$ can be understood as the set of (finite) reduced words
in $\mathcal{A}^{\pm 1}$.
Then the boundary $\partial \FN$ is naturally identified to 
the set of (right) infinite reduced words 
$X=x_1\dots x_p\dots$ 
with $x_i\in\mathcal{A}^{\pm 1}$, $x_{i+1}\neq x_i^{-1}$.
The cylinder defined by a reduced word $w=w_1\dots w_p$ 
is the set of right-infinite reduced words 
$X=x_1\dots x_k\dots$ which admit $w$ as prefix:
$x_i=w_i$ for $i\in\{1,\dots,p\}$.
A basis of topology of $\partial \FN$ is given by the set of
all such cylinders.

An automorphism $\phi$ of a
free group $\FN$ induces a homeomorphism
$\partial\phi$ of the boundary $\partial \FN$. 
This can easily be checked by considering a standard set 
of generators of the
automorphisms group $\Aut(\FN)$ of $\FN$.
Alternatively, this can be seen as a consequence of the
fact that a quasi-isometry of a proper Gromov-hyperbolic
space induces a homeomorphism on the boundary of
this space, see \cite{Ghys}. 
Indeed, $\FN$ equipped with the word metric associated
to a basis $\mathcal{A}$, is a proper Gromov-$0$-hyperbolic 
space, and any automorphism of
$\FN$ is a quasi-isometry of $\FN$ with respect to this metric.

\subsection{Compactification of $\FN$}\label{sec:compactification}

Let $\BFN$ denote the union of $\FN$ and
its boundary $\partial\FN$, i.e. $\BFN=\FN\cup\partial\FN$.
Given a basis of $\FN$,
if $w$ is a reduced word, let $C_w$ be the set of
reduced finite or infinite words
which have $w$ as prefix.
A basis of topology of $\BFN$ is given by 
the finite subsets of $\FN$ and the sets $C_w$
(with $w$ describing all the reduced words of $\FN$).
Then $\BFN$ is a compact set,
and the inclusions of $\FN$ and $\partial\FN$
in $\BFN$ are embeddings.
If $\phi$ is an automorphism of $\FN$, 
$\overline{\phi}$ will denote the map defined by 
$\overline{\phi}(g)=\phi(g)$ if $g\in\FN$ and
$\overline{\phi}(X)=\partial\phi(X)$ if $X\in\partial\FN$.
The map $\overline{\phi}$ is a homeomorphism of $\BFN$.


\subsection{Getting rid of periodicity}\label{sec:periodicity}

Let $f$ be a homeomorphism of a topological space $\mathcal{X}$.
We denote by $\Fix(f)=\{x\in \mathcal{X} \;|\;  f(x)=x\}$ the set
of fixed points of $f$, and by 
$\Per(f)=\bigcup_{k\in\N}\Fix(f^k)$
the set of periodic points of $f$.

Levitt and Lustig have proved in \cite{LL2} that there exists
an integer $p$, which depends only on the rank $N$ of $\FN$,
such that for all $\phi\in\Aut(\FN)$, the periodic points of
$\overline{\phi}^p$ are fixed points: 
$\Fix(\overline{\phi}^p)=\Per(\overline{\phi}^p)$.
This result has been refined by Feighn and Handel 
in \cite{FH}, where the notion of ``forward rotationless''
outer automorphism has been introduced.
This lead us to say, in this paper, that an automorphism 
$\phi\in\Aut(\FN)$ is {\em rotationless} if
$\Fix(\overline{\phi})=\Per(\overline{\phi})$.
The previously mentioned result can be rephrased as follows:

\begin{theorem}[Levitt-Lustig]
Any automorphism $\phi\in\Aut(\FN)$ has
a power $\phi^p$ ($p\in\N$) which is rotationless.
\end{theorem}


\subsection{Nature of fixed points}\label{sec:nature}

Let $\phi$ be a rotationless automorphism of $\FN$.
The set $\Fix(\phi)$ is a subgroup of $\FN$, which is called
the {\em fixed subgroup} of $\phi$.
This fixed subgroup has finite rank, see \cite{Coo}. More precisely, it is
proved in \cite{BH} that $\rank(\Fix(\phi))\leq N$. 
In particular, $\Fix(\phi)$ is a quasiconvex subgroup of $\FN$,
and thus its boundary $\partial\Fix(\phi)$ naturally injects into
$\partial\FN$.
By continuity of $\overline{\phi}$, every point of 
$\partial\Fix(\phi)$ is contained in $\Fix(\partial\phi)$.
Following Nielsen, 
these fixed points of $\partial\phi$
are called \textit{singular}; the fixed points of $\partial\phi$ which
are not singular are called \textit{regular}.

A fixed point $X$ of $\partial\phi$ is 
\textit{attracting} if there exists a neighbourhood $U$ of $X$ in
$\overline{F}_N$ such that the sequence $\overline{\phi}^k(x)$
converges to $X$ for all $x$ in $U$.
A fixed point $X$ of $\partial\phi$ is 
\textit{repulsing} if it is attracting for $\partial{\phi}^{-1}$.
It is proved in \cite{GJLL} that:

\begin{lemma}\label{lem:regular}
Let $\phi\in\Aut(F_N)$. A regular fixed point of $\partial\phi$ is
either attracting or repulsing.
\qed
\end{lemma}

However, outside of the regular fixed point set, i.e. for singular fixed points, the dynamics can be quite a bit more complicated. In particular, there may exist {\em mixed} fixed points, i.e. fixed points which serve as attractor for some orbits, and simultaneously as repeller for others.  This phenomenon is rather common; some concrete examples will be spelled out in the subsequent sections.

A particular case of a mixed fixed point is the case (defined in the Introduction) of a parabolic fixed point.  Thus we obtain as special case the following consequence of Lemma \ref{lem:regular}:

\begin{remark}\label{rem:singular}
Any parabolic fixed point
of $\phi$ is singular.
\end{remark}

\subsection{Limit points}\label{sec:limit}

Let $\phi$ be a rotationless automorphism of $\FN$.
For any $x\in\BFN$, if the limit 
$\lim_{k\rightarrow+\infty}\overline{\phi}^k(x)$
exists,
we denote it by $\omega_{\phi}(x)$.
In \cite{LL3}, Levitt and Lustig have proved:

\begin{theorem}[Levitt-Lustig]\label{th:asymptotic}
Let $\phi\in\Aut(\FN)$ be rotationless.
Then for any $x\in\BFN$ the sequence
$\overline{\phi}^{k}(x)$ converges to some element
$\omega_{\phi}(x)\in\Fix(\overline{\phi})$.
\qed
\end{theorem}

A point $X\in\partial\Fix(\phi)$ is a {\em $\omega$-limit point} of $\phi$
if there exists $x\in\BFN$ such that 
$X=\omega_{\phi}(x)$.
A point $X\in\partial\Fix(\phi)$ is a {\em limit point} of $\phi$
if it is a $\omega$-limit point of $\phi$ or $\phi^{-1}$.
Let $L^\omega_\phi$ denote the set of 
$\omega$-limit points of $\phi$ and 
let $L_\phi$ denote the set of 
limit points of $\phi$.

For any $g\in\FN$, $g\neq 1$, the sequence $g^k$ has a limit
in $\partial\FN$ when $k\rightarrow+\infty$: this limit is
denoted by $g^\infty$.

\begin{proposition}\label{prop:omega}
Let $\phi\in\Aut(\FN)$ be a rotationless automorphism.
If $g\in\FN\smallsetminus\Fix(\phi)$,
then:
$$\omega_\phi(g)=\omega_{\phi}(g^\infty).$$
\end{proposition}

\begin{proof}
The proof is a simple adaptation of the arguments in the proof 
of \cite[Proposition 2.3]{LL1}. We fix a basis $\mathcal{A}$ of $\FN$.
We note that for all $g\in\FN\smallsetminus\{1\}$,
the Gromov product $(g,g^\infty)$
(i.e. the length of longest common prefix)
of $g$ and $g^\infty$ is bigger than $\frac{1}{2}(|g|+1)$
(where $|g|$ denotes the length of $g$ in the basis $\mathcal{A}$).
If $g\notin\Fix(\phi)$, then the length 
of $\phi^k(g)$,
and thus also the Gromov product $(\phi^k(g),(\phi^k(g))^\infty)$,
tend to infinity.
Theorem \ref{th:asymptotic} implies that 
$\omega_\phi(g)=\omega_{\phi}(g^\infty)$.
\end{proof}

Proposition \ref{prop:omega} shows that 
$L^\omega_\phi=\{\omega_\phi(X)\;|\; X\in\partial\FN\}$.
We do not know whether
$L^\omega_\phi=\{\omega_\phi(g)\;|\; g\in\FN\}$ holds.

\subsection{Isoglossy classes}\label{sec:isoglossy}

For any $\phi\in\Aut(\FN)$, 
two points $X,Y\in\partial\FN$ are called {\em isogloss}
(with respect to $\phi$) if there exists some $g\in\Fix(\phi)$
such that $X=gY$. 
It follows directly from this definition that 
isoglossy is an equivalence
relation. 
The fixed subgroup $\Fix(\phi)$ acts naturally on the
fixed point set
$\Fix(\partial\phi)$, which is thus naturally 
partitioned into isoglossy classes.
If $X,Y\in\Fix(\partial\phi)$ are isogloss, then they
are of same ``dynamical type'': they are simultaneously
singular, attracting, repulsing, mixed, 
parabolic or limit points.

\subsection{Dynamics graph}\label{sec:graph}

Let $\phi\in\Aut(\FN)$ be a rotationless automorphism. 
We associate to $\phi$ a graph $\Gamma_\phi$, called
the {\em dynamics graph} of $\phi$.
The vertices of $\Gamma_\phi$ are the isoglossy classes
of points of $L_{\phi}$. There is an oriented edge from the 
isoglossy class $x_1$ to the isoglossy class $x_2$ if there exists some representatives
$X_i$ of $x_i$ and $X\in\partial \FN$ such that
$\omega_{\phi^{-1}}(X)=X_1$ and
$\omega_{\phi}(X)=X_2$.
The main theorem of \cite{H0} states that
$\Gamma_\phi$ is a finite graph. 
We give in Figure \ref{fig:ns} the dynamics graph of 
an automorphism which has North-South dynamics
on $\partial\FN$.

\begin{figure}[h]
\begin{center}
$ \xymatrix{ 
\bullet \ar[d] \\
\bullet
}$
\caption{North-South dynamics graph} \label{fig:ns}
\end{center}
\end{figure}
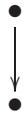

Finally, we note that, for a rotationless automorphism $\phi$,
the existence of parabolic orbit is equivalent to the
fact that there is an edge of the dynamics graph $\Gamma_\phi$ 
which is a loop.

\begin{remark}
In \cite{Lev} G. Levitt introduces a graph 
in order to code the dynamics of 
so-called ``simple-dynamics homeomorphisms'' of the Cantor set $C$: 
a homeomorphism $f:C\rightarrow C$ has {\em simple dynamics} if 
the set $\Fix(f)$ of its fixed points is finite,
and if the sequence $f^n$ uniformly converges on any compact set
disjoint from $\Fix(f)$. 
If $\phi\in\Aut(\FN)$ is a rotationless automorphism
with trivial fixed subgroup, 
then $\partial\phi^p$ has simple dynamics, 
and the graph $\Gamma_{\phi}$ is the same as the one defined in
\cite{Lev}.  In this case, the fixed points of $\partial\phi$ are 
either attracting or repulsing.
Thus, if one is interested in parabolic orbits, which are the main focus of 
the present paper, one has to purposefully leave to world of 
``simple dynamics'' homeomorphisms.
\end{remark}


\section{Examples}\label{sec:examples}

\subsection{Inner automorphisms}\label{sec:inner}

Let $i_u\in\Aut(\FN)$ denote the {\em conjugation}, 
or {\em inner automorphism}, by $u\in \FN$, i.e.  
$i_u(g)=ugu^{-1}$ for all $g\in \FN$. 
The set $\Inn(\FN)$ of inner automorphisms of $\FN$
is a normal subgroup of $\Aut(\FN)$. The quotient
group, denoted by $\Out(\FN)$, is the group of
{\em outer automorphisms} of $\FN$.

The homeomorphism $\partial i_u:\partial \FN \rightarrow \partial \FN$
induced by $i_u$ is the left translation by $u$: 
$\partial i_u(X)=uX$.
If $u\neq 1$, the map $\partial i_u$ has precisely 2 fixed points:
$u^\infty$ and $u^{-\infty}$ (where $u^\infty$ is the limit
of the sequence $u^k$, and $u^{-\infty}$ is is the limit of 
the sequence $u^{-k}$, for $k\rightarrow +\infty$).
Moreover, for any point $X\in\partial\FN$ different from 
$u^{-\infty}$, the sequence $\partial\phi^k(X)$ converges
to $u^\infty$ when $k$ tends to infinity.
One checks easily that the map $\partial i_u$ has North-South dynamics,
from $u^{-\infty}$ to $u^\infty$, on $\partial\FN$,
see \cite{LL1} for instance.

\begin{remark}
We note that the fixed subgroup of $i_u$ is cyclic,
generated by the root of $u$ (i.e. the element $v\in\FN$
such that $u=v^p$ with $p\in\N$ maximal).
In particular, $u^\infty$ and $u^{-\infty}$ 
are singular fixed points of $i_u$.
This shows that
when defining ``$X$ is an attracting fixed point of $\phi$''
in \S \ref{sec:nature}, 
it makes a crucial difference that we request the neighbourhood 
$U$ of $X$ to be taken in $\BFN$ 
and not just in $\partial\FN$.
\end{remark}


\subsection{Geometric automorphisms}\label{sec:geometric}

Let $\Sigma$ be a compact surface with fundamental group $\pi_1(\Sigma)$
isomorphic to $\FN$ (in particular, $\Sigma$ has non empty boundary).
The surface $\Sigma$ can be equipped with a hyperbolic metric (i.e. a metric
of constant curvature equal to $-1$) in such a way that
every boundary component of the boundary of $\Sigma$ is a geodesic.
The universal cover $\tilde{\Sigma}$ of $\Sigma$ is then identified with a closed
convex subset of the hyperbolic plane $\Hyp^2$, and
the Gromov boundary $\partial\tilde{\Sigma}$ of $\tilde{\Sigma}$,
which is naturally identified with the boundary $\partial\FN$ of $\FN$, injects in
the boundary (or circle at infinity) $S_\infty$ of $\Hyp^2$. 
Since $S_\infty$ is a circle, it can be equipped with a natural cyclic order.
This order on $S_\infty$ induces a cyclic order on $\partial\FN$.

In his fundamental work \cite{Nie1,Nie2,Nie3}, Nielsen proposed an original
and fruitful point of view to study homeomorphisms of surfaces. The basic idea
is that the behaviour of a homeomorphism $f$ of 
a surface $\Sigma$ is well reflected by 
the collection of all the lifts $\tilde{f}$ of $f$ to $\tilde{\Sigma}$
which have each much simpler individual behaviour. 
This idea is at the origin of what is now called ``Nielsen-Thurston classification''
of homeomorphisms of surfaces, see \cite{HT},
and it has much influenced the study of (outer) automorphisms of free
groups, see \cite{GJLL,FH,HM}.
The key fact is that
any lift $\tilde{f}$ of $f$ induces a homeomorphism $\partial\tilde{f}$
of $\partial\tilde{\Sigma}$.
A basic (but rather fundamental) remark is that
$\partial\tilde{f}$ preserves the cyclic order on 
$\partial\tilde{\Sigma}\subseteq S_\infty$.

An homeomorphism $f$ of $\Sigma$ induces an outer automorphism of 
$\pi_1(\Sigma)$, and thus an outer automorphism $\Phi\in\Out(\FN)$
(in fact, this outer automorphism $\Phi$ only depends on the mapping class
of $f$). Such an outer automorphism $\Phi$ of $\FN$ 
(and also any automorphism $\phi\in\Phi$) 
is called {\em geometric}.
Classical Galois theory for covering spaces states that the lifts of $f$ are in
bijective correspondance with the automorphisms in the outer class $\Phi$.
More precisely, an automorphism $\phi\in\Phi$ and a lift $\tilde{f}$ of $f$ are 
in correspondance if, and only if, 
$$\phi(g)\circ \tilde{f}=\tilde{f}\circ g \quad\forall g\in F_N,$$ 
where the elements of $F_N$ are considered as deck
transformations of $\tilde{\Sigma}$. 
As a consequence, the dynamics of $\partial\tilde{f}$ on $\partial\tilde{\Sigma}$
and the dynamics of $\partial\phi$ on $\partial \FN$ 
are conjugated via the natural identification
between $\partial\tilde{\Sigma}$ and $\partial \FN$.

It follows from the previous discussion that, for any geometric automorphism
$\phi\in\Aut(\FN)$, the homeomorphism $\partial\phi$ of $\partial\FN$
must preserve a cyclic order on $\partial\FN$.

Another 
fact proved by Nielsen is that $\partial\tilde{f}$ has at least 2 periodic
points on $\partial\FN$ (for a proof in the context of free groups, see \cite{LL2}).
This means that there exists a positive power of $\partial\tilde{f}$ which has
at least 2 fixed points on $\partial\FN$.
Both these facts (existence of 2 fixed points and preservation of a cyclic
order) yield directly:

\begin{proposition}
\label{no-parabolics-in-geometrics}
A geometric automorphism of $\FN$ can not have a parabolic
orbit in $\partial\FN$.
\end{proposition}

This fact is particularly meaningful for the free group of 
rank 2. Indeed, it is well known that any outer automorphism of $F_2$
can be induced by a homeomorphism of a torus with one boundary component,
see \cite{Nie0}.
This is precisely how  
Proposition \ref{prop:F2}
is proved.

\subsection{Outer automorphisms}
\label{sec:outer-autos}
Although well known, we believe that at this point it might be wise 
to alert the less expert reader about a common misunderstanding.
It is by no means true that any two automorphisms $\phi, \phi'$ which 
belong to the same outer automorphism class $\Phi$, 
must have conjugated dynamics. Indeed, their dynamics graphs 
$\Gamma_\phi$ and $\Gamma_{\phi'}$ may look quite different. 
Concrete examples are easy to come by, and some are given in the 
subsequent sections.

The reader who wants to be more subtle can easily check that 
indeed some automorphisms in $\Phi$ have naturally conjugated dynamics. 
The resulting {\em isogredience classes} go again all the way back to Nielsen
(see also \cite{LL1}), 
and one could associate to $\Phi$ a {\em total dynamics graph} 
which is the disjoint union of the $\Gamma_\phi$ over a set of representatives 
for the single isogredience classes. However, this goes beyond the scope of this paper.


\section{Parabolic orbits}\label{sec:parabolic}

\subsection{Structure of a parabolic fixed point}\label{sec:structure}

Let $\phi\in\Aut(\FN)$ be an automorphism, 
and $X\in\Fix(\partial\phi)$ be a parabolic fixed point for $\phi$.
We have seen (cf Remark \ref{rem:singular}) that $X$
must be singular. 
A point $X\in\partial\FN$ is {\em rational} if it a fixed
point of an inner automorphism, i.e. $X=u^\infty$
for some $u\in\FN\smallsetminus\{1\}$.
It is proved in \cite{H0} that singular limit points of $\phi$
are rational. 
We deduce the following:

\begin{lemma}
A parabolic fixed point $X$ of $\phi\in\Aut(\FN)$ is a singular 
rational point:
$X=u^\infty$ with $u\in\Fix(\phi)$.
\qed
\end{lemma}

Moreover, we have:

\begin{proposition}
Let $\phi$ be an automorphisme of $\FN$, and $X\in\Fix(\partial\phi)$ 
be a parabolic fixed point for $\phi$. 
Then any neighborhood of $X$ in $\partial \FN$ contains
a full orbit 
$\{\partial\phi^k(Y)\; |\; k\in\Z \} \subset \partial\FN$
\end{proposition}

\begin{proof}
We have seen that $X=u^{\infty}$, with $u\in\Fix(\phi)$.
We consider a given neighborhood $\mathcal{V}$ of $X$.
Let $\vartheta=\{\partial\phi^k(Y))\; |\; k\in\Z\}$ 
be a parabolic orbit for $X$. 
We note that $\vartheta\cup\{X\}$ is a compact subset of $\partial\FN$.
Moreover, $u^{-\infty}\notin\vartheta\cup\{X\}$ because $Y\notin\Fix(\partial\phi)$.
Since the sequence $(\partial i_u^p)_{p\in\N}$ uniformly converges on compact subsets
of $\partial\FN\smallsetminus\{u^{-\infty}\}$
towards $u^\infty$ when $p$ tends to infinity,
see \S \ref{sec:inner},
the set $\partial i_u^p (\vartheta)$ is contained in
$\mathcal{V}$, up to taking $p$
sufficiently large. We remark that, since
$u\in\Fix(\phi)$, $\partial i_u^p (\partial \phi^k(Y)) 
= \partial \phi^k(u^pY)$, and thus 
$\partial i_u^p (\vartheta)=u^p\vartheta$ is a parabolic orbit for
$X$.
\end{proof}

\subsection{Automorphisms of $F_4$ which have parabolic orbits}

For any $k\in\N$, consider the automorphism $\phi_k$ of $F_4=<a,b,c,d>$ given by: $$
\begin{array}{crcl}
\phi_k: & a & \mapsto & a \\  
& b & \mapsto & ba \\ 
& c & \mapsto & ca^{k+1} \\ 
& d & \mapsto & dc
\end{array}
$$ and its inverse: $$
\begin{array}{crcl}
\phi^{-1}_k: & a & \mapsto & a \\  
& b & \mapsto & ba^{-1} \\ 
& c & \mapsto & ca^{-k-1} \\ 
& d & \mapsto & da^{k+1}c^{-1}
\end{array}
$$

The rose $R_4$ is the geometric realization of
graph with one vertex and 4 edges. 
We put an orientation on each
edge, and we label them by $a$, $b$, $c$ and
$d$. We can turn $R_4$ into a length space
by declaring that each edge has length 1.
As usual, the automorphisms $\phi_k^{\pm1}$
can be realized  as homotopy equivalences 
$f_k^{\pm}$ of the rose $R_4$
where each edge is mapped linearly to the edge path
with label preassigned by $\phi_k^{\pm1}$.

In fact, the automorphisms $\phi_k^{\pm1}$
define outer automorphisms which are 
unipotent polynomially growing in the sense
of \cite{BFH1}, and the maps $f_k^+$ satisfy
the conclusions of Theorem 5.1.8 of \cite{BFH1}.
We do not quote here the statement of this
theorem, which would lead us to introduce a lot 
of technical background, but we freely use in the 
sequel some consequences of it.

Let $\mathcal{A}$ be a basis of $\FN$. We denote
by $[g]$ the reduced word, in the basis $\mathcal{A}$,
representing the element $g\in\FN$.
Let $\phi$ be an automorphism of $\FN$.
A {\em splitting} of $g\in\FN$ for $\phi$ is a way to write
$g=g_1 \ldots g_n$ 
such that:
\begin{enumerate}[(i)]
\item $n\geq 2$,
\item  for all $i\in\{1,\ldots,n\}$, $g_i\in\FN\smallsetminus\{1\}$,
\item for all $p\in\N$, for all $i\in\{1,\ldots,n-1\}$,
$[\phi^p(g_i)][\phi^p(g_{i+1})]=[\phi^p(g_ig_{i+1})]$
(this means that no cancellation occurs between
$[\phi^p(g_i)]$ and $[\phi^p(g_{i+1})]$).
\end{enumerate}
In that case, we note $g=g_1\cdot \ldots \cdot g_n$,
and each $g_i$ is called a {\em brick} of the splitting.

We now apply that Theorem 5.1.8 of \cite{BFH1} to the given
family $\phi_k$ and obtain:

\begin{lemma}\label{lem:split}
For all $g\in F_4$, there exists some $p_0\in\N$ such that for all $p\geq p_0$,
$[\phi_k^p(g)]$ and $[\phi_k^{-p}(g)]$ have a splitting, the bricks of which are either edges
or paths of the following labels: $ba^qb^{-1}$, $ca^qc^{-1}$, 
$ba^qc^{-1}$ or $ca^qb^{-1}$,  for some $q\in\Z$.
\qed
\end{lemma}

\begin{remark}
For the reader who is familiar with the terminology of \cite{BFH1},
the edge paths labelled by $ba^qb^{-1}$, $ca^qc^{-1}$, 
$ba^qc^{-1}$ or $ca^qb^{-1}$ are precisely the exceptional paths
of the improved train-track map $f_k$.
\end{remark}

As a consequence of Lemma \ref{lem:split}, one can easily check that the 
sequence $(|[\phi_k^p(g)]|)_{p\in\N}$
of  lengths of $[\phi_k^p(g)]$ is bounded above
by a polynomial of degree 2 in $p$.

It is claimed in \cite{Mas} that there exists a general algorithm to compute 
the fixed subgroup of a given automorphism of $\FN$.
There exist some easier algorithms for special cases: for instance, one
could use \cite{CL} to compute the fixed subgroup of $\phi_k$.
In fact, it is sufficient to determines the so called {\em indivisible
Nielsen paths}, see \cite{BH}: using Lemma \ref{lem:split},
 we find that
$\Fix(\phi_k)=\;<a,bab^{-1},cac^{-1}>$.

\begin{theorem}
The set $\{\phi_k\;|\;k\in\N\}$ is
a family of automorphisms of $F_4$,
such that each $\phi_k$ has a parabolic orbit.
The dynamics graph of ${\phi_k}$ is given in Figure
\ref{fig:parabolique}.
For any $k,k',p,p'\in\N$, $\phi_k^p$ and 
$\phi_{k'}^{p'}$ are conjugated if and only if
$k=k'$ and $p=p'$.
\end{theorem}

\begin{figure}[ht!]
\begin{center}
\centerline{
\xy 
(0,0)*+{ba^{-\infty}}; (-2,-3)
**\crv{(-10,10)&(-15,0)&(-10,-10)}
?>*\dir{>},
(2,3); (20,0)*+{ba^{+\infty}}
**\crv{(10,15)}
?>*\dir{>},
(18,-3); (2,-3)
**\crv{(10,-15)}
?>*\dir{>},
(-16,0)*+{^{bd^{-1}}},
(10,10)*+{^{b}},
(10,-12)*+{^{bc^{-1}}},
\endxy
}
$ \xymatrix{ 
X_k^+ & X_k^- \ar[l]_{d} & a^{+\infty}  \ar[r]^{b^{-1}}& 
a^{-\infty} & ca^{-\infty} \ar[r]^{c} \ar[l]_{d^{-1}} & ca^{+\infty}
}$
\caption{The dynamics graph of $\phi_k$ has 3 connected
components.
A label $g$ has been added to each edge: it means
that $\omega_k(g)$ is the endpoint of the edge
and $\omega_k^-(g)$ is the origin of the edge.
(The terminology used here is given in the proof below.)
} \label{fig:parabolique}
\end{center}
\end{figure}
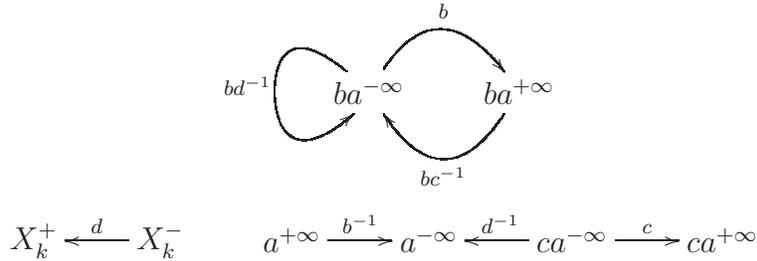

\begin{proof}
For simplicity, we write $\omega_{\phi_k}=\omega_k$
and  $\omega_{\phi_k^{-1}}=\omega_k^-$.
Using Lemma \ref{lem:split}, we check that 
$\phi_k$ has only one isoglossy class of attracting fixed points:
a representative is given by 
$X^+_k=\omega_k(d)=dcca^{k+1}ca^{2k+2}ca^{3k+3}\ldots$
Likewise $\phi_k$ has only one isoglossy class of repulsing fixed points:
a representative is given by 
$X^-_k=\omega_k^-(d)=da^{k+1}c^{-1}a^{2k+2}c^{-1}a^{3k+3}c^{-1}\ldots$

Lemma \ref{lem:split} also gives:
\begin{itemize}
\item $\omega_k(b^{-1})=\omega_k(c^{-1})=\omega_k(d^{-1})=
a^{-\infty}$,
\item $\omega_k(c)=ca^{+\infty}$,
\item $\omega_k(b)=ba^{+\infty}$,
\item $\omega_k(bc^{-1})=ba^{-\infty}$,
\item $\omega_k^-(b^{-1})=\omega_k^-(c^{-1})=a^{+\infty}$,
\item $\omega_k^-(c)=\omega_k^-(d^{-1})=ca^{-\infty}$,
\item $\omega_k^-(b)=ba^{-\infty}$,
\item $\omega_k^-(bc^{-1})=ba^{+\infty}$.
\end{itemize}
In fact, one can see that there are only 5 isoglossy classes in
$L_{\phi_k}^\omega$ (given by $X^+_k$, $a^{-\infty}$, $ca^{+\infty}$,
$ba^{+\infty}$, $ba^{-\infty}$) and 5 isoglossy classes in
$L_{\phi_k^{-1}}^\omega$ (given by $X^-_k$, $a^{+\infty}$, $ca^{-\infty}$,
$ba^{-\infty}$, $ba^{+\infty}$).

Note that $\phi_k(bd^{-1})=bac^{-1}\cdot d^{-1}$ is a splitting
for $\phi_k$. Hence
$\omega_k(bd^{-1})=\omega_k(bac^{-1})=ba^{-\infty}$. On the
other hand, $b\cdot d^{-1}$ is a splitting for $\phi_k^{-1}$. Hence
$\omega_k^-(bd^{-1})=\omega_k^-(b)=ba^{-\infty}$.
Thus $ba^{-\infty}$ is parabolic fixed point for $\phi_k$.

Suppose that $\phi_k^p$ and $\phi_{k'}^{p'}$ are conjugated
($k,k',p,p'\in\N$): there exists
$\psi\in\Aut(F_4)$ such that $\phi_k^p=\psi\phi_{k'}^{p'}\psi^{-1}$. 
Let $M_k, M_{k'}, P\in GL(4,\Z)$ be the matrices obtained by abelianization of 
respectively $\phi_k$, $\phi_{k'}$ and $\psi$. Then 
$$
M_k^p=\left(\begin{array}{cccc} 
1 & p & (k+1)p & \frac{1}{2}(k+1)p(p-1) \\ 
0 & 1 & 0 & 0 \\
0 & 0 & 1 & p \\
0 & 0 & 0 & 1
\end{array}\right).
$$
Computing $M_k^pP=PM_{k'}^{p'}$, one sees that $P$ must have the following shape:
$$
P=\left(\begin{array}{cccc} 
\lambda_1 & \mu_1& \mu_2 & \mu_3 \\ 
0 & \lambda_2 & 0 & \mu_4 \\
0 & \mu_5 & \lambda_3 & \mu_6 \\
0 & 0 & 0 & \lambda_4
\end{array}\right)
$$
with 
\begin{equation}\label{eq:pk}
p'(k'+1)\lambda_3=p(k+1)\lambda_1
\;\;\text{ and }\;\; p'\lambda_3=p\lambda_4.
\end{equation}
We deduce that $\det P=\lambda_1\lambda_2\lambda_3\lambda_4$, and thus
$\lambda_i\in\{\pm1\}$, since $\det P=\pm1$. 
From (\ref{eq:pk}) we derive $k=k'$ and $p=p'$. 
\end{proof}

\subsection{Parabolic orbits for $N \geq 5$}

For any $k\in\N$, consider the automorphism $\alpha_k$ of $F_5=<a,b,c,d,e>$ given by: $$
\begin{array}{ccccc}
\alpha_k: & a & \mapsto & a \\  
& b & \mapsto & ba \\ 
& c & \mapsto &
ca^{k+1} \\ 
& d & \mapsto & dc \\
& e & \mapsto & e
\end{array}
$$ 
Since the restriction of $\alpha_k$ to $<a,b,c,d>$ is $\phi_k$,
it is clear that 
$\omega_{\alpha_k}(bd^{-1})=\omega_{\alpha_k^{-1}}(bac^{-1})=ba^{-\infty}$
is a parabolic fixed point for $\alpha_k$.
Considering the abelianization and
arguing as previously, we check that if $k\neq k'$
and $p\neq p'$, then
$\alpha_k^p$ and $\alpha_{k'}^{p'}$ can not be conjugated.

If $N\geq 6$, we split $\FN=F_4*F_2*F_{N-6}$.
We first recall some facts about $\Out(F_2)$.
It is well known, since Nielsen \cite{Nie0}, 
that the abelianisation morphism 
from $\Out(F_2)$ to $GL_2(\Z)$ is an isomorphism.
If $M\in SL_2(\Z)$ has a trace bigger than 2,
then $M$ has an eigenvalue $\lambda>1$ which is
an algebraic unity of a quadratic extension of $\Q$:
we call $\lambda$ the dilatation of $M$.
For all $k\in\N$ prime, there exists $M_k\in SL_2(\Z)$
such that the dilatation $\lambda_k$ of $M_k$ belongs to 
$\Q(\sqrt{k})\smallsetminus \Q$. This implies in 
particular that for all $p\in\N$, $\lambda_k^p\in
\Q(\sqrt{k})\smallsetminus \Q$.
We choose $\theta_k\in\Aut(F_2)$ in the outer class
represented by $M_k$.
Then the automorphism $\theta_k^p$ has growth rate
equal to $\lambda_k^p$. 

We define $\beta_k\in\Aut(\FN)$ by
$\beta_k=\phi_1*\theta_k*id$, where $id$ is the
identity on $F_{N-6}$. Again,
$\omega_{\beta_k}(bd^{-1})=\omega_{\beta_k^{-1}}(bac^{-1})=ba^{-\infty}$
is a parabolic fixed point for $\beta_k$.
Since $\phi_1$ is polynomially growing, it follows that the
growth rate of $\beta_k^p$ is $\lambda_k^p$ (see for instance \cite{Lev1}).
This proves that $\beta_k^p$ is not conjugated to $\beta_{k'}^{p'}$
if $k\neq k'$ or $p\neq p'$, because the growth rate is a conjugacy
invariant
and because $\Q(\sqrt{k})\cap\Q(\sqrt{k'})=\Q$ (if $k$ and $k'$
are prime integers).

This finishes the proof Theorem \ref{thm:main}.
In view of Proposition \ref{prop:F2}, it remains
to ask the following question, the answer of which
we do not know:

\begin{question}
Does there exist an automorphism of $F_3$
which has a parabolic orbit?
\end{question}


\section{Dehn twist automorphisms of $F_2$}\label{sec:dehn}

In this last section, we calculate the dynamics graphs of all the
automorphisms in the outer class of $\delta^n$ ($n\in\Z$, $n\neq 0$), 
where $\delta$ is the automorphism of $F_2=<a,b>$ defined by
$\delta(a)=a$ and $\delta(b)=ba$.

Let $D\in\Out(F_2)$ be the outer class of $\delta$.
As explained in \S \ref{sec:geometric}, the automorphisms
in the outer class $D^n$ ($n\in\Z$) can not have parabolic orbits.
We are going to describe more precisely the dynamics induced on 
$\partial\FN$ by the automorphisms in the outer class $D^n$ ($n\in\Z$, $n\neq 0$).
For that, we pursue the strategy of \cite{GJLL,LL1}, where the interested
reader will be able to find details of the following constructions.

The rose $R_2$ is the geometric realization of the
graph with one vertex and $2$ edges. 
We put an orientation on each
edge, and we label them by $a$ and $b$. 
We can turn $R_2$ in a length space
by declaring that each edge has length 1.
We represent $D^n$ by an homotopy equivalence $f$
of $R_2$ defined  in the following way: 
$f$ is the identity on the edge $a$
and linearly sends the edge $b$ to the edge path
labelled $ba^n$.

The universal cover $\widetilde{R}_2$ of $R_2$ is a tree,
equipped by the action of $F_2$ by deck transformations.
 We lift the labels of the edges of $R_2$ to the edges of
 $\widetilde{R}_2$. Equivalently, $\widetilde{R}_2$ can be
 considered as the Cayley graph of $F_2$ relative to the
 generating set $\{a,b\}$.
 Let $T$ be the tree obtained by contracting in $\widetilde{R}_2$
 all the edges labelled by $a$: the action of $\FN$ on 
 $\widetilde{R}_2$ induces an action of $F_2$ on $T$ by isometries.
 We note that the stabilizer of a vertex of $T$ is conjugated to
 the subgroup $<a>\;\subset\FN$ generated by $a$.

As in the geometric case (see \S \ref{sec:geometric}) 
the automorphisms in the outer class $D^n$ are in 1:1 correspondance
with the lifts of $f$ to $\widetilde{R}_2$.
Moreover, 
these lifts of $f$ induce isometries of $T$.
More precisely, the isometry $H$ of $T$ associated to the automorphism 
$\delta^n\in D^n$ satisfies
$$\delta^n(g)\circ H=H\circ g \quad\forall g\in
F_N,$$ where the elements of $F_N$ are considered as isometries
of $T$.
Then, for $u\in\FN$, the map 
$H_u=u\circ H$ is the isometry of $T$ associated 
to the automorphism  $i_u\circ\delta^n\in D^n$, since
$(i_u\circ \delta^n)(g)\circ H_u=H_u\circ g$ 
holds for all $g\in F_N$. 

If $H_u$ is a hyperbolic isometry of $T$, then  $i_u\circ\delta^n$
has North-South dynamics and the fixed points of $i_u\circ\delta^n$
are determined by the ends of the axis of $H_u$ in $T$, see \cite{LL1}.

If $H_u$ is an elliptic isometry, let $P\in T$ be a fixed point of $H_u$.
There exists some $w\in\FN$ such that the stabilizer of $P$ in $\FN$
is $w<a>w^{-1}$. The fact that $P$ is a fixed point of $H_u$ then
results in the existence of an integer $k\in\Z$ such that
$u\delta^n(w)=wa^k$.
Or equivalently, such that  
$i_u\circ\delta^n = i_w\circ(i_{a^k}\circ\delta^n)\circ i_w^{-1}$.
Indeed, 
$$\begin{array}{rcl}
i_u\circ\delta^n & = & i_{w a^k (\delta^n(w))^{-1} }\circ\delta^n \\
 & = & i_{w{a^k}{\delta^n(w^{-1})}}\circ\delta^n \\
 & = & i_w\circ i_{a^k}\circ i_{\delta^n(w^{-1})}\circ\delta^n \\
 & = & i_w\circ i_{a^k}\circ\delta^n\circ i_{w^{-1}}. 
\end{array}$$
The dynamics of $\partial(i_u\circ\delta^n)$ is thus conjugated
to the dynamics of $\partial(i_{a^k}\circ\delta^n)$ for some $k\in\Z$.
We are now going to study in more detail the automorphisms 
$i_{a^k}\circ\delta^n$ for $k\in\Z$, and in particular, to give their
dynamics graphs.

The inverse of $i_{a^k}\circ\delta^n$ is $i_{a^{-k}}\circ\delta^{-n}$.
We note that:
$$
\begin{array}{crclccrcl}
i_{a^k}\circ \delta^n: & a & \mapsto & a & & i_{a^{-k}}\circ \delta^{-n}: & a & \mapsto & a \\ 
& b & \mapsto & a^{k}ba^{n-k} & & & b & \mapsto & a^{-k}ba^{k-n} \\ 
& b^{-1} & \mapsto & a^{k-n}b^{-1}a^{-k} & & & b^{-1} & \mapsto & a^{n-k}b^{-1}a^{k}.
\end{array}
$$
Thus the dynamics of $\partial(i_{a_2^k}\circ\delta^n)$ depends on the
sign of $k$ and of $n-k$.

\begin{remark}\label{rem:sigma}
Let $\sigma\in\Aut(\FN)$ defined by $\sigma(a)=a^{-1}$ 
and $\sigma(b)=b^{-1}$.
We note that $i_{a^k}\circ\delta^n$ and  $i_{a^{n-k}}\circ\delta^n$ 
are conjugated by the involution $\sigma$.
\end{remark}

\noindent{\bf First case:} Assume $k(n-k)=0$.
Since $\delta^n$ and $i_{a^{n}}\circ\delta^n$ are conjugated
by $\sigma$ (see Remark \ref{rem:sigma}), we focus on $\delta^n$.
One can check that $\Fix(\delta^n)=<a,bab^{-1}>$.
Let $X$ be a point in $\partial F_2\smallsetminus\partial<a,bab^{-1}>$,
and let $x$ be the longest prefix of $X$ in $<a,bab^{-1}>$.
Then $X=xY$, with no cancellation between $x$ and $Y$, 
and the first letter of $Y$ is equal to $b$ or to $b^{-1}$.
If $Y$ begins by $b$, then $\omega_{\delta^n}(Y)=ba^{\infty}$
and $\omega_{\delta^{-n}}(Y)=ba^{-\infty}$.
If $Y$ begins by $b^{-1}$, then $\omega_{\delta^n}(Y)=a^{-\infty}$
and $\omega_{\delta^{-n}}(Y)=a^{\infty}$.
Hence $\delta^n$ has 2 isoglossy classes of $\omega$-limit points
(with representatives $ba^\infty$ and $a^{-\infty}$),
and $\delta^{-n}$ has 2 isoglossy classes of $\omega$-limit points
(with representatives $ba^{-\infty}$ and $a^{\infty}$).
The dynamics graph of $\delta^n$ is given in Figure \ref{fig:k(n-k)=0}.

\begin{figure}[ht!]
\begin{center}
$ \xymatrix{ 
ba^\infty &  a^{-\infty} & &\ar@{--}[dd] & & b^{-1}a^{-\infty}  & a^{\infty} \\
ba^{-\infty} \ar[u] & a^{\infty} \ar[u] & & & & b^{-1}a^{\infty} \ar[u] & a^{-\infty} \ar[u] \\
  k=0 &&&&&& k=n 
}$
\caption{Dynamics graph of $i_{a^k}\circ\delta^n$
for $k(n-k)=0$.
} \label{fig:k(n-k)=0}
\end{center}
\end{figure}

\noindent{\bf Second case:} $k(n-k)<0$.
We suppose that $k>n$ (from which one deduces the case $k<0$ by using 
Remark \ref{rem:sigma}).
The fixed subgroup is $\Fix (i_{a^k}\circ\delta^n) =\; <a>$.
We note that 
$\omega_{i_{a^k}\circ\delta^n}(b)=
\omega_{i_{a^k}\circ\delta^n}(b^{-1})=a^\infty$
and 
$\omega_{(i_{a^{k}}\circ\delta^{n})^{-1}}(b)=
\omega_{(i_{a^{k}}\circ\delta^{n})^{-1}}(b^{-1})=a^{-\infty}$.
It follows that  $\partial(i_{a^{k}}\circ\delta^{n})$ has 
North-South dynamics on $\partial F_2$,
see Figure \ref{fig:k(n-k)<0}.

\begin{figure}[ht!]
\begin{center}
$ \xymatrix{ 
a^{-\infty} & &\ar@{--}[dd] & & a^{\infty} \\
a^{\infty} \ar[u] & & & & a^{-\infty} \ar[u] \\
  k<0 &&&& k>n 
}$
\caption{Dynamics graph of $i_{a^k}\circ\delta^n$
for $k(n-k)<0$.
} \label{fig:k(n-k)<0}
\end{center}
\end{figure}

\noindent{\bf Third case:} $k(n-k)>0$, i.e. $0<k<n$.
We check that the fixed subgroup is equal to $\Fix (i_{a^k}\circ\delta^n) =\; <a>$.
We note that 
$\omega_{i_{a^k}\circ\delta^n}(b)=a^\infty$,
$\omega_{i_{a^k}\circ\delta^n}(b^{-1})=a^{-\infty}$,
$\omega_{(i_{a^{k}}\circ\delta^{n})^{-1}}(b)=a^{-\infty}$ and
$\omega_{(i_{a^{k}}\circ\delta^{n})^{-1}}(b^{-1})=a^{\infty}$.
For $x\in\partial F_2$, 
it follows that $\omega_{i_{a^k}\circ\delta^n}(X)$
and $\omega_{(i_{a^{k}}\circ\delta^{n})^{-1}}(X)$
depend only on the first occurence of the letter $b$
or $b^{-1}$ in $X$:
if it is $b$, then 
$\omega_{i_{a^k}\circ\delta^n}(X)=a^\infty$
and $\omega_{(i_{a^{k}}\circ\delta^{n})^{-1}}(X)=a^{-\infty}$;
if it is $b^{-1}$, then 
$\omega_{i_{a^k}\circ\delta^n}(X)=a^{-\infty}$
and $\omega_{(i_{a^{k}}\circ\delta^{n})^{-1}}(X)=a^{\infty}$.
We say that $\partial(i_{a^{k}}\circ\delta^{n})$ has 
{\em semi-North-South dynamics} on $\partial F_2$,
see Figure \ref{fig:k(n-k)>0}.

\begin{figure}[ht!]
\begin{center}
$ \xymatrix{ 
a^{-\infty} \ar@/^1pc/[d]\\
a^{\infty} \ar@/^1pc/[u]
}$
\caption{Dynamics graph of $i_{a^k}\circ\delta^n$
for $k(n-k)>0$.
} \label{fig:k(n-k)>0}
\end{center}
\end{figure}


\bibliographystyle{amsplain}
\bibliography{biblio}

\end{document}